\documentclass[a4paper, 12pt, final]{amsart}
\usepackage{amscd}
\usepackage{mathrsfs}
\usepackage{}
\usepackage[all]{xy}
\usepackage[OT1]{fontenc}
\usepackage[latin1]{inputenc}
\usepackage{amssymb,latexsym}
\usepackage{type1cm,ae}
\usepackage[notref,notcite]{showkeys}
\usepackage{graphicx}
\usepackage{xspace}
\usepackage{setspace}
\usepackage[english]{babel}

\theoremstyle{definition}
		\newtheorem{theorem}{Theorem}[section]
				
				\newtheorem{lemma}[theorem]{Lemma}

                \newtheorem{example}[theorem]{Example}

	            \newtheorem{remark}[theorem]{Remark}
\numberwithin{equation}{section}

\newcommand*{\bR}{\ensuremath{\mathbb{R}}}

\newcommand*{\bdary}[1]{\partial #1}

\newcommand*{\Wert}{\mathord{\mbox{|\kern-1.5pt|\kern-1.5pt|}}}
\newcommand*{\ie}{\mbox{i.e.}\xspace}

\DeclareMathOperator{\dist}{dist}
\DeclareMathOperator{\diam}{diam}

\DeclareMathOperator{\capacity}{Cap}

\def\Xint#1{\mathchoice
  {\XXint\displaystyle\textstyle{#1}}%
  {\XXint\textstyle\scriptstyle{#1}}%
  {\XXint\scriptstyle\scriptscriptstyle{#1}}%
  {\XXint\scriptscriptstyle\scriptscriptstyle{#1}}%
  \!\int}
\def\XXint#1#2#3{{\setbox0=\hbox{$#1{#2#3}{\int}$}
  \vcenter{\hbox{$#2#3$}}\kern-.5\wd0}}

\def\dashint{\Xint-}

\title[Sharp capacity estimates in s-John domains]{Sharp capacity estimates in s-John domains}
\author{Chang-Yu Guo}
\address[Chang-Yu Guo]{Department of Mathematics and Statistics, University of Jyv\"askyl\"a, P.O. Box 35, FI-40014 University of Jyv\"askyl\"a, Finland}
\email{changyu.c.guo@jyu.fi}

\thanks{C.Y.Guo was partially supported by the Academy of Finland grant 131477 and the Magnus Ehrnrooth foundation.}

\begin{document}

\begin{abstract}
It is well-known that several problems related to analysis on $s$-John domains can be unified by certain capacity lower estimates. In this paper, we obtain general lower bounds of $p$-capacity of a compact set $E$ and the central Whitney cube $Q_0$ in terms of the Hausdorff $q$-content of $E$ in an $s$-John domain $\Omega$. Moreover, we construct several examples to show the essential sharpness of our estimates.
\end{abstract}

\maketitle
\section{Introduction}\label{sec:first}
Recall that a bounded domain $\Omega\subset \bR^n$ is a John domain if there
is a constant $C$ and a point $x_0\in \Omega$ so that, for each $x\in \Omega,$
one can find a rectifiable curve $\gamma:[0,1]\to \Omega$ with $\gamma(0)=x,$
$\gamma(1)=x_0$ and with
\begin{equation} \label{eka}
    Cd(\gamma(t),\bdary\Omega)\ge l(\gamma([0,t]))
\end{equation}
for each $0<t\le 1.$ F. John used this condition in his work on
elasticity~\cite{j61} and the term was coined by
Martio and Sarvas~\cite{ms79}. Smith and Stegenga \cite{ss90} introduced
the more general concept of $s$-John domains, $s\ge 1,$ by replacing \eqref{eka}
with
\begin{equation} \label{toka}
    Cd(\gamma(t),\bdary\Omega)\ge l(\gamma([0,t]))^s.
\end{equation}
The condition~\ref{eka} is called a ``twisted cone condition" in literature. Thus condition~\ref{toka} should be called a ``twisted cusp condition". 

In the last twenty years, $s$-John domains has been extensively studied in connection with Sobolev type inequalities; see~\cite{bk95,hk98,hk00, km00, kot02, ss90}. In particular, Buckley and Koskela~\cite{bk95} have shown that a simply connected planar domain which supports a Sobolev-Poincar\'e inequality is an $s$-John domain for an appropriate $s$. Smith and Stegenga have shown that an $s$-John domain $\Omega$ is a $p$-Poincar\'e domain, provided $s<\frac{n}{n-1}+\frac{p-1}{n}$. In particular, if $s<\frac{n}{n-1}$, then $\Omega$ is a $p$-Poincar\'e domain for all $1\leq p<\infty$. These results were further generalized to the case of $(q,p)$-Poincar\'e domains in~\cite{hk98, km00, kot02}. Recall that a bounded domain $\Omega\subset \bR^n$, $n\geq 2$, is said to be a $(q,p)$-Poincar\'e domain if there exists a constant $C_{q,p}=C_{q,p}(\Omega)$ such that
\begin{equation}\label{def:poincare domain}
\Big(\int_\Omega |u(x)-u_\Omega|^q dx\Big)^{1/q}\leq C_{q,p}\Big(\int_\Omega |\nabla u(x)|^p dx\Big)^{1/p}
\end{equation} 
for all $u\in C^\infty(\Omega)$. Here $u_\Omega=\dashint_{\Omega}u(x)dx$. When $q=p$, $\Omega$ is termed a $p$-Poincar\'e domain and when $q>p$ we say that $\Omega$ supports a Sobolev-Poincar\'e inequality.

The recent studies \cite{aim09,g13b,gkt12} on mappings of finite distortion
have generated new interest in the class of $s$-John domains. In particular, uniform continuity of quasiconformal mappings onto $s$-John domains was studied in~\cite{g13,gk13}. 

The proofs for both types of problems rely on certain capacity estimates for subsets of $s$-John domains. To be more precise, for the problem related to Sobolev-Poincar\'e inequalities, one uses the idea of Maz'ya~\cite{m60,m11} to reduce the problem to capacity estimates of the form 
\begin{align*}
\capacity_p(E,Q_0,\Omega)\geq \psi(|E|),
\end{align*}
where $Q_0$ is the fixed Whitney cube containing the (John) center $x_0$ and  $E$ is an admissible subset of $\Omega$ disjoint from $Q_0$; for~\eqref{def:poincare domain}, $\psi(t)=Ct^{p/q}$, see also~\cite{hak98,kot02}. Here, by admissible we mean that $E$ is an open set so that $\bdary E\cap \Omega$ is a smooth  submanifold. As for the uniform continuity of quasiconformal mappings onto $s$-John domains, one essentially needs a capacity estimate of the form
\begin{align*}
\capacity_n(E,Q_0,\Omega)\geq \psi(\diam E),
\end{align*}
where $E$ is a continuum in $\Omega$ disjoint from the central Whitney cube $Q_0$; see~\cite{g13}. Thus one could expect that a more general capacity estimate of the form 
\begin{align}\label{eq:expected}
\capacity_p(E,Q_0,\Omega)\geq \psi(\mathcal{H}_\infty^q(E))
\end{align} 
holds in certain $s$-John domains $\Omega$,  
where $E$ is a compact set in $\Omega$ disjoint from the central Whitney cube $Q_0$ and $\mathcal{H}_\infty^q(E)$ is the Hausdorff $q$-content of $E$. We confirm this expectation by showing the following result.

\begin{theorem}\label{thm:main theorem}
Let $\Omega\subset\bR^n$, $n\geq 2$, be an $s$-John domain. For $0<\varepsilon<1$, $1\leq p\leq n$ and  $q\geq s(n-1)+1-p+\varepsilon$, there exists a positive constant $C(n,p,q,s,\varepsilon)$ such that
\begin{equation}\label{eq:sharp capacity}
\capacity_p(E,Q_0,\Omega)\geq C(n,p,q,s,\varepsilon)\Big(\mathcal{H}_\infty^q(E)\Big)^{\frac{s(n-1)+1-p+\varepsilon}{q}},
\end{equation}
whenever $E\subset\Omega$ is a compact set disjoint from $Q_0$. 

\end{theorem}

\begin{remark}\label{rmk:for previous results}
If $p=n$, $q=1$ and $E\subset\Omega$ is a continuum, then~\eqref{eq:sharp capacity} reduces to the estimate
\begin{equation*}
\capacity_n(E,Q_0,\Omega)\geq C(n,s,\varepsilon) (\diam E)^{(n-1)(s-1)+\varepsilon}.
\end{equation*}
The restriction becomes $1\geq (s-1)(n-1)+\varepsilon$, which is equivalent to $s\leq 1+\frac{1-\varepsilon}{n-1}$. The range for $s$ is essentially sharp, see~\cite{gk13}.

If $q=n$, then~\eqref{eq:sharp capacity} reduces to the estimate 
\begin{equation*}
\capacity_p(E,Q_0,\Omega)\geq C(n,s,\varepsilon) |E|^{\frac{(n-1)s+1-p+\varepsilon}{n}}.
\end{equation*}
The restriction becomes $s\leq 1+\frac{p-\varepsilon}{n-1}$. Note that
$$1+\frac{p}{n-1}> \frac{n}{n-1}+\frac{p-1}{n}.$$ 
This implies that if $s<1+\frac{p}{n-1}$, then $\Omega$ is a $p$-Poincar\'e domain. The range for $s$ is sharp, see~\cite{hk98}.
\end{remark}

The estimate in Theorem~\ref{thm:main theorem} is essentially sharp in the sense that the exponent of $\mathcal{H}_\infty^q(E)$ in~\eqref{eq:sharp capacity} cannot be made strictly smaller than $\frac{s(n-1)+1-p}{q}$; see Example~\ref{example:sharpness} below.

Our second result shows that the requirement $q\geq s(n-1)+1-p+\varepsilon$ is essentially sharp in the sense that there exists an $s$-John domain $\Omega\subset\bR^n$ such that no estimate of the form as in~\eqref{eq:expected} holds in $\Omega$ whenever $q<\min\{s(n-1)+1-p,n\}$. 
This is somewhat surprising since the estimate in~\eqref{eq:sharp capacity} does not degenerate when $q<s(n-1)+1-p$. 

\begin{theorem}\label{example}
Fix $1\leq p\leq n$. There exists an $s$-John domain $\Omega\subset \bR^n$ such that there is a sequence of compact sets $E_j$ in $\Omega$ with the following properties: 
\begin{itemize}
\item Each $E_j$ is disjoint from the  central Whitney cube $Q_0$;
\item $\mathcal{H}_\infty^q(E_j)$ is  bounded from below uniformly by a positive constant and $\capacity_p(E_j,Q_0,\Omega)\to 0$ as $j\to \infty$,  whenever $q<\min\{(n-1)s+1-p,n\}$.
\end{itemize}
\end{theorem}  

%
%

It would be interesting to know whether one can obtain an estimate of the form as in~\eqref{eq:expected} when $q=(n-1)s+1-p$.

When $q<\min\{(n-1)s+1-p,\log_2(2^n-1)\}$, the $s$-John domain $\Omega$ constructed in Theorem~\ref{example} is in fact Gromov hyperbolic in the quasihyperbolic metric. This is very surprising, since it was proven in~\cite{g13} that for all Gromov hyperbolic $s$-John domains $\Omega$, an estimate of the form as in~\eqref{eq:expected} holds when $p=n$, $q=1$ and $E\subset \Omega$ is a continuum. Our example shows that one can not replace the assumption being a continuum by just being compact, and still obtain the estimate for all $s$-John domains. For definitions and examples of Gromov hyperbolic domains, we refer to the beautiful monograph~\cite{bhk01}.

\section{Preliminary results}

For an increasing function $\tau:[0,\infty)\to[0,\infty)$ with $\tau(0)=0$, we denote by $\mathcal {H}^{\tau}_\infty$ the Hausdorff $\tau$-content: $\mathcal{H}^\tau_\infty(E)=\inf\sum_i\tau(r_i)$, where the infimum is taken over all coverings of $E\subset\bR^n$ with balls $B(x_i,r_i)$, $i=1,2,\dots$ When $\tau(t)=t^s$ for some $0<s<\infty$, we write $\mathcal{H}^s_\infty=\mathcal{H}^\tau_\infty$.

For disjoint compact sets $E$ and $F$ in the domain $\Omega$, we denote by $\capacity_p(E,F,\Omega)$ the $p$-capacity of the pair $(E,F)$:
\begin{equation*}
    \capacity_p(E,F,\Omega)=\inf_u \int_{\Omega}|\nabla u(x)|^pdx,
\end{equation*}
where the infimum is taken over all continuous functions $u\in W_{loc}^{1,p}(\Omega)$ which satisfy $u(x)\leq 0$ for $x\in E$ and $u(x)\geq 1$ for $x\in F$.

Let $\Omega$ be a bounded domain in $\bR^n$, $n\geq 2$. Then $\mathbb{W}=\mathbb{W}(\Omega)$ denotes a Whitney decomposition of $\Omega$, \ie a collection of closed cubes $Q\subset\Omega$ with pairwise disjoint interiors and having edges parallel to the coordinate axes, such that $\Omega=\cup_{Q\in \mathbb{W}}Q$, the diameters of $Q\in \mathbb{W}$ belong to the set $\{2^{-j}:j\in \mathbb{Z}\}$ and satisfy the condition
\begin{equation*}
\diam(Q)\leq \dist(Q,\bdary\Omega)\leq 4\diam(Q).
\end{equation*}
For $j\in\mathbb{Z}$ we define
\begin{equation*}
\mathbb{W}_j=\{Q\in \mathbb{W}:\diam(Q)=2^{-j}\}.
\end{equation*}

The following lemma is well-known, see for instance~\cite[Lemma 2.8]{k12}.
\begin{lemma}\label{lemma:bojarski}
Fix $1\leq p<\infty$. Let $B_1$, $B_2$, $\dots$ be balls or cubes in $\bR^n$, $a_j\geq 0$ and $\lambda>1$. Then
\begin{align*}
\|\sum a_j\chi_{\lambda B_j}\|_p\leq C(\lambda,n,p)\|\sum a_j\chi_{ B_j}\|_p
\end{align*}
\end{lemma}

\section{Main proofs}

\begin{proof}[Proof of Theorem~\ref{thm:main theorem}]
The proof is a combination of several well-known arguments; in particular~\cite[Proof of Theorem 9]{hak98} and~\cite[Proof of Theorem 5.9]{hk98}. For any compact set $E\subset \Omega$ such that $E\cap Q_0=\emptyset$, where $Q_0$ is the central cube that contains the John center $x_0$, we fix a test function $u$ for $\capacity_p(E,Q_0,\Omega)$, \ie $u$ is a continuous function in $W^{1,p}_{loc}(\Omega)$ so that $u\geq 1$ on  $E$ and $u\leq 0$ on $Q_0$. We may assume that $\diam\Omega=1$.

For each $x\in E$, we may fix an $s$-John curve $\gamma$ joining $x$ to $x_0$ in $\Omega$ and define $P(x)$ to be the collection of Whitney cubes that intersect $\gamma$. Thus $Q(x)\in P(x)$ will be the Whitney cube containing the point $x$. We next divide our compact set $E$ into the good part and the bad part according to the range of $u_{Q}$. 
Let $\mathscr{G}=\{x\in E:u_{Q(x)}\leq \frac{1}{2}\}$ and $\mathscr{B}=E\backslash \mathscr{G}$. 


\textit{Claim 1:} for $1\leq p\leq n$ and $q\geq s(n-1)+1-p+\varepsilon$, there exists a positive constant $C(n,p,q,s,\varepsilon)$ such that
\begin{equation}\label{eq.claim3}
\int_{\Omega} |\nabla u(x)|^pdx\geq C(n,p,q,s,\varepsilon)\Big(\mathcal{H}_\infty^q(\mathscr{B})\Big)^{\frac{s(n-1)+1-p+\varepsilon}{q}}.
\end{equation}
\textit{Proof of Claim 1:} 
Fix $1\leq p\leq n$, $q\geq s(n-1)+1-p+\varepsilon$ and set $\Delta=\frac{\varepsilon}{2}$. 
%
%
%
Let $Q_i, i=1,\dots,m$ be those Whitney cubes that intersect $\mathscr{B}$. 
Fix one such Whitney cube $Q_{i_0}$ and let $x_{i_0}$ be its center.
Let $Q_{i_0}^j,j=1,\dots,k$ be the Whitney cubes in $P(x_{i_0})$ with $Q_{i_0}^k=Q_{i_0}$. The standard chaining argument involving Poincar\'e inequality~\cite{ss90} gives us the estimate
\begin{equation*}
1\lesssim \sum_{j=1}^k \diam Q_{i_0}^j\dashint_{Q_{i_0}^j}|\nabla u(y)|dy.
\end{equation*}
H\"older's inequality implies
\begin{equation*}
1\lesssim \Big(\sum_{j=0}^kr_j^{(1-\kappa)p/(p-1)} \Big)^{(p-1)/p}\Big(\sum_{j=0}^kr_j^{\kappa p-n}\int_{Q_{i_0}^j}|\nabla u|^p \Big)^{1/p},
\end{equation*}
where $r_j=\diam Q_{i_0}^j$ and $\kappa=\frac{s+p-1-\Delta}{sp}$. Using the $s$-John condition, one can easily conclude
\begin{equation*}
\sum_{j=0}^kr_j^{(1-\kappa)p/(p-1)}<C.
\end{equation*}
Therefore,
\begin{equation}\label{eq:hk23}
\sum_{j=0}^kr_j^{\kappa p-n}\int_{Q_{i_0}^j}|\nabla u|^p \geq C,
\end{equation}
where the constant $C$ depends only on $p$, $n$, $\Delta$ and the constant from the $s$-John condition.

By the $s$-John condition $Cr_j\geq |x-y|^s$, for $y\in Q_{i_0}^j$, and since $\kappa p-n<0$ according to our choice $p\leq n$, we obtain
\begin{equation*}
r_j^{\kappa p-n}\lesssim |x-y|^{s(\kappa p-n)}
\end{equation*}
for $y\in Q_{i_0}^j$. For $y\in Q_{i_0}^i\cap (2^{j+1}Q_{i_0}\backslash 2^{j}Q_{i_0})$, we have $|x-y|\approx 2^jr_k$ and hence for such $y$,
\begin{equation}\label{eq:hk24}
r_i^{\kappa p-n}\lesssim (2^jr_k)^{s(\kappa p-n)}.
\end{equation}
Combining~\eqref{eq:hk23} with~\eqref{eq:hk24} leads to
\begin{align*}
1&\lesssim \sum_{j=0}^kr_j^{\kappa p-n}\int_{Q_{i_0}^j}|\nabla u|^p\lesssim (r_k)^{s(\kappa p-n)}\int_{Q_{i_0}}|\nabla u|^p\\
&+\sum_{j=0}^{|\log r_k|}(2^jr_k)^{s(\kappa p-n)}\int_{(2^{j+1}Q_{i_0}\backslash 2^{j}Q_{i_0})\cap\Omega}|\nabla u|^p\\
&\lesssim \sum_{l=0}^{|\log r_k|+1}(2^lr_k)^{s(\kappa p-n)}\int_{2^lQ_{i_0}\cap\Omega}|\nabla u|^p.
\end{align*}
On the other hand,
\begin{equation*}
\sum_{l=0}^{|\log r_k|+1}(2^lr_k)^\Delta <r_k^\Delta\sum_{l=-\infty}^{|\log r_k|+1}2^{l\Delta}<C.
\end{equation*}
Combining the above two estimates, we conclude that there exists an $l$ (depending on $\Delta$ and hence $\varepsilon$) such that
\begin{equation*}
(2^lr_k)^\Delta\lesssim (2^lr_k)^{s(\kappa p-n)}\int_{2^lQ_{i_0}\cap\Omega}|\nabla u|^p.
\end{equation*}
It follows that,
\begin{equation*}
\int_{\Omega\cap 2^lQ_{i_0}}|\nabla u|^p\gtrsim (2^lr_k)^{s(n-\kappa p)+\Delta}=(2^lr_k)^{s(n-1)+1-p+\varepsilon}.
\end{equation*}
In other words, there exists an $R_x\geq d(x,\bdary\Omega)/2$ with
\begin{equation*}
\Big(\int_{\Omega\cap B(x,R_x)}|\nabla u|^p\Big)^{\frac{q}{s(n-1)+1-p+\varepsilon}}\gtrsim R_x^q.
\end{equation*}
Applying the Vitali covering lemma to the covering $\{B(x,R_x)\}_{x\in E}$ of the set $\mathscr{B}$, we can select pairwise disjoint balls $B_1,\dots,B_k,\dots$ such that $\mathscr{B}\subset \bigcup_{i=1}^\infty 5B_i$. Let $r_i$ denote the radius of the ball $B_i$. Then
\begin{align*}
\mathcal{H}^q_\infty(\mathscr{B})&\leq \sum_{i=1}^\infty(\diam 5B_i)^q=5^q\sum_{i=1}^\infty r_i^q\\
&\lesssim \sum_{i=1}^\infty\Big(\int_{\Omega\cap B_i}|\nabla u|^p\Big)^{\frac{q}{s(n-1)+1-p+\varepsilon}}
\end{align*}
The desired capacity estimate follows by noticing the elementary inequality
\begin{equation*}
\sum_i a_i^b \lesssim \Big(\sum_i a_i\Big)^b, \qquad b\geq 1.
\end{equation*}

\textit{Claim 2:} for $n-q<p\leq n$ and  $0<\varepsilon<p+q-n$,
\begin{equation}\label{eq:claim1}
\int_{\Omega} |\nabla u(x)|^pdx\geq C(p,q,n,\varepsilon)\Big(\mathcal{H}_\infty^q(\mathscr{G})\Big)^{\frac{n-p+\varepsilon}{q}}.
\end{equation}

\textit{Proof of Claim 2:} Fix $n-q<p\leq n$ and $0<\varepsilon<p+q-n$. 
Our aim is to show that
\begin{equation}\label{eq:key estimate}
\int_{2Q(x)}|\nabla u(x)|^pdx\geq C(p,s,n)\mathcal{H}_\infty^s(\mathscr{G}\cap Q(x))
\end{equation}
for any $n-p<s\leq n$. We adapt the argument from~\cite[Proof of Theorem 5.9]{hk98}. 

Fix $n-p<s\leq n$. For $y\in \mathscr{G}$, $u_{Q(y)}\leq \frac{1}{2}$. For $x\in \mathscr{G}\cap Q(y)$, write $Q_i=Q(x,r_i)$, where $r_i=2^{-i-1}\diam Q(y)$. Then
\begin{align*}
u(x)=\lim_{i\to \infty}u_{Q_i}=\lim_{i\to\infty}\dashint_{Q_i}u.
\end{align*}
Now 
\begin{align*}
\frac{1}{2}\leq |u(x)-u_{Q_0}|\leq \sum_{i\geq 0}|u_{Q_i}-u_{Q_{i+1}}|.
\end{align*}
Since by the Poincar\'e inequality
\begin{align*}
|u_{Q_i}-u_{Q_{i+1}}|\leq C(n)r_i^{\frac{p+s-n}{p}}\Big(r_i^{-s}\int_{Q_i}|\nabla u|^p\Big)^{\frac{1}{p}},
\end{align*}
we obtain that 
\begin{align*}
\frac{1}{2}&\leq \sum_{i=1}^\infty C(n)r_i^{\frac{p+s-n}{p}}\Big(r_i^{-s}\int_{Q_i}|\nabla u|^p\Big)^{\frac{1}{p}}\\
&\leq C(p,s,n)(\diam Q(y))^{\frac{p+s-n}{p}}\sup_{0<t\leq \diam Q(y)}\Big(t^{-s}\int_{Q(x,t)}|\nabla u|^p\Big)^{\frac{1}{p}}\\
&\leq C(p,s,n)\sup_{0<t\leq \diam Q(y)}\Big(t^{-s}\int_{Q(x,t)}|\nabla u|^p\Big)^{\frac{1}{p}}.
\end{align*}
Thus, for each $x\in \mathscr{G}\cap Q(y)$, there is a cube $Q(x,t_x)$ such that $t_x\leq \diam Q(y)$ and that 
\begin{align*}
t_x^s\leq C(p,s,n)\int_{Q(x,t_x)}|\nabla u|^p.
\end{align*}
By Vitali we can find pairwise disjoint cubes $Q_1$, $Q_2$, $\dots$ as above such that $\mathscr{G}\cap Q(y)\subset \bigcup 5Q_i$. Then
\begin{align*}
\mathcal{H}_\infty^s(\mathscr{G}\cap Q(y))&\leq C(p,s,n)\sum_{i=1}^\infty \int_{Q_i}|\nabla u|^p\\
&\leq C(p,s,n)\int_{2Q(y)}|\nabla u|^p.
\end{align*}
Thus the proof of~\eqref{eq:key estimate} is complete. 

We next show that for $n-q<p\leq n$ and for fixed $0<\varepsilon<p+q-n$, the following estimate holds.
\begin{align}\label{eq:consequence of key est}
\int_{2Q(x)}|\nabla u(x)|^pdx\geq C(p,q,n,\varepsilon)\Big(\mathcal{H}_\infty^q(\mathscr{G}\cap Q(x))\Big)^{\frac{n-p+\varepsilon}{q}}
\end{align}

Let $\varepsilon>0$ be as above. We set $s=n-p+\varepsilon$. Then $s<q$. Now~\eqref{eq:consequence of key est} follows from~\eqref{eq:key estimate} and the trivial estimate 
\begin{align*}
\Big(\mathcal{H}_\infty^q(E)\Big)^{\frac{s}{q}}\lesssim \mathcal{H}_\infty^s(E).
\end{align*}

Taking into account the sub-additivity of Hausdorff $q$-content and concavity of the function $t\mapsto t^{\frac{n-p+\varepsilon}{q}}$, ~\eqref{eq:claim1} follows immediately from~\eqref{eq:consequence of key est} and Lemma~\ref{lemma:bojarski}.


\end{proof}

\section{Examples}

\begin{example}\label{example:sharpness}
We will use the standard ``rooms and corridors" type domains. This type of domains consists of a central cube shaped room  along with an infinite disjoint collection of  cube shaped rooms which are connected to the central room by narrow cylindrical corridors; see Figure~\ref{fig:room and corridor}. 

For each $j\in \mathbb{N}$, the attached cube shaped room $E_j$ is of edge length $r_j$ and the narrow cylindrical corridor is of radius $r_j^s$ and height $r_j$. We can ensure that the rooms and corridors are pairwise disjoint by requiring the sequence $\{r_j\}_{j\in \mathbb{N}}$ to decrease to zero sufficiently rapidly. It is clear that $\Omega$ is an $s$-John domain.

\begin{figure}[h]
  \includegraphics[width=8cm,height=10cm]{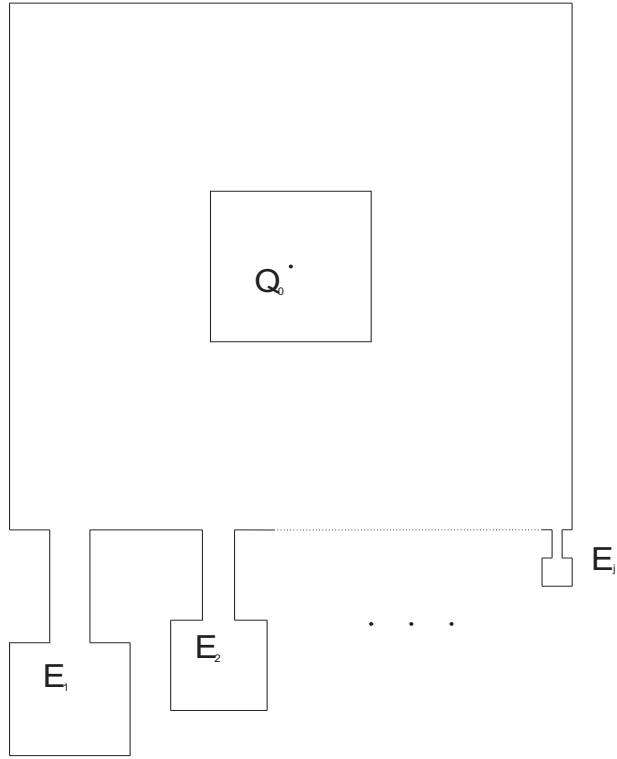}\\
  \caption{The standard ``room and corridors" type domain}\label{fig:room and corridor}
\end{figure}

For $s<\frac{p+q-1}{n-1}$, 
we may choose $\varepsilon>0$ such that $q\geq s(n-1)+1-p+\varepsilon$. Then it is easy to obtain the following estimate:
\begin{align*}
\capacity_p(E_j,Q_0,\Omega)\leq Cr_j^{(n-1)s-p+1}\leq C\mathcal{H}_\infty^q(E_j)^{\frac{(n-1)s-p+1}{q}}
\end{align*}
Noticing that $r_j\to 0$ as $j\to \infty$, this implies that the exponent of $\mathcal{H}_\infty^q(E)$ in Theorem~\ref{thm:main theorem} is essentially best possible. 

\end{example}

\begin{example}\label{example:q big}
Fix  $p\in [1,n]$, $n\geq 2$. There exists an $s$-John domain $\Omega$ in $\bR^n$ such that there is a sequence of compact sets $E_j$ in $\Omega$  with the following two properties: 
\begin{itemize}
\item Each $E_j$ is disjoint from the central Whitney cube $Q_0$;
\item $\mathcal{H}_\infty^q(E_j)$ is  bounded from below uniformly by a positive constant and $\capacity_p(E_j,Q_0,\Omega)\to 0$ as $j\to \infty$, whenever $n-1\leq q<\min\{(n-1)s+1-p,n\}$. 
\end{itemize}

The idea of the construction of such an $s$-John domain is the following: we first construct a John domain $\Omega_0$ such that the number $N_j$ of Whitney cubes of size (comparable to) $r_j=2^{-j}$ in $\Omega_0$ is approximately $2^{qj}$. We then build a ``room and $s$-passage" $Q_s$ in each Whitney cube $Q\subset\Omega_0$ and $Q\neq Q_0$, where $Q_0$ is the central Whitney cube $Q_0$ containing the John center; see Figure~\ref{fig:room and s-passage}. If the Whitney cube $Q$ is of edge length $4r_j$, then the attached room shaped cube is of side length $r_j$ and the corresponding $s$-passage is of radius $r_j^s$ and height $r_j$.

\begin{figure}[h]
  \includegraphics[width=10cm]{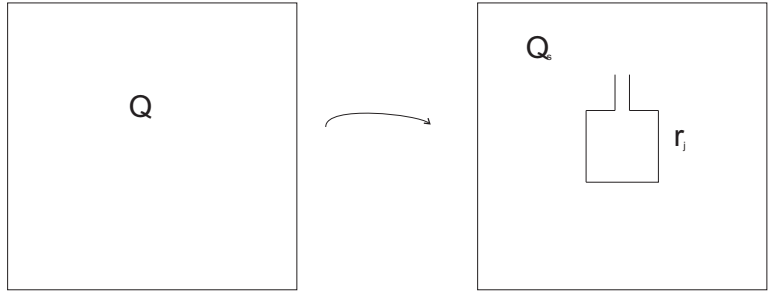}\\
  \caption{``room and $s$-passage" type replacement}\label{fig:room and s-passage}
\end{figure}

Let $E_j$ be the union of all the room shaped cube of edge length $r_j$. Then we have the following trivial upper estimate
\begin{align*}
\capacity_p(E_j,Q_0,\Omega)\leq CN_j\cdot r_j^{(n-1)s-p+1}\leq Cr_j^{(n-1)s-p-q+1}.
\end{align*}
Thus $\capacity_p(E_j,Q_0,\Omega) \to 0$ whenever $q<(n-1)s-p+1$. On the other hand, noting that all the cubes in $E_j$ are well separated, to estimate the Hausdorff $q$-content, one has to cover each such cube by a ball of the same size (since otherwise the ball will intersects two cubes and substantially increases the radius). Thus we have
\begin{align*}
\mathcal{H}_\infty^q(E_j)\geq CN_j\cdot r_j^q\geq C.
\end{align*}

To construct a John domain with the desired property, one essentially needs to construct a John domain $\Omega_0$ such that $\dim_{\mathcal{M}}(\bdary\Omega_0)=q$ when $q\in [n-1,n)$, where $\dim_{\mathcal{M}}$ denotes the upper Minkowski dimension. With this understood, one can select certain Von Koch type curve as the boundary of a John domain; see~\cite[Proposition 5.2]{hhv13} for the detailed construction of such a John domain $\Omega_0$. It is clearly that the ``room and $s$-passage" type replacement described above turns $\Omega_0$ into an $s$-John domain $\Omega$. In fact, $\dim_{\mathcal{M}}(\bdary\Omega_0)=\dim_{\mathcal{M}}(\bdary\Omega)=q$. For these facts, see~\cite[Proposition 5.11 and Proposition 5.16]{hhv13}. 

\end{example}

\begin{example}\label{example:q small}
Fix $1\leq p\leq n$. There exists an $s$-John domain, which is Gromov hyperbolic in the quasihyperbolic metric, such that there is a sequence of compact sets $E_j$ in $\Omega$ with the follow properties: 
\begin{itemize}
\item Each $E_j$ is disjoint from the  central Whitney cube $Q_0$;
\item $\mathcal{H}_\infty^q(E_j)$ is  bounded from below uniformly by a positive constant and $\capacity_p(E_j,Q_0,\Omega)\to 0$ as $j\to \infty$,  whenever $q<\min\{(n-1)s+1-p,\log_2(2^n-1)\}$.
\end{itemize}

We first give a detailed construction of the $s$-John domain $\Omega$ in the plane with the desired properties. Fix $1\leq p\leq 2$. We first consider the case $q=\log_23$. The $s$-John domain $\Omega$ will be constructed by an inductive process. In the first step, we have a unit cube $Q$ and four ``room and $s$-passage" type ``legs" as in Figure~\ref{fig:room and s-passage}. The ``$s$-passage" $R_1$ is a rectangle of length $2^{-1}$ and width $2^{-s-1}$ and the ``room" $Q_1$ is a cube of edge-length $2^{-1}$. In the second step, we attach at each of the three corners of $Q_1$ a ``room and $s$-passage" type ``legs". The ``$s$-passage" $R_2$ is a rectangle of length $2^{-2}$ and width $2^{-2s-1}$ and the ``room" $Q_2$ is a cube of edge-length $2^{-2}$. In general at step $j$, we have $4\cdot 3^{j-1}$ ``room and $s$-passage" type ``legs", where the ``$s$-passage" $R_j$ is a rectangle of length $2^{-j}$ and width $2^{-js-1}$ and the ``room" $Q_j$ is a cube of edge-length $2^{-j}$. It is easy to check that, with our choices of parameters, there is no overlap in our construction. Moreover, $\Omega$ is an $s$-John domain that is Gromov hyperbolic in the quasihyperbolic metric (since $\Omega$ is simply connected). 

\begin{figure}[h]
  \includegraphics[width=9cm]{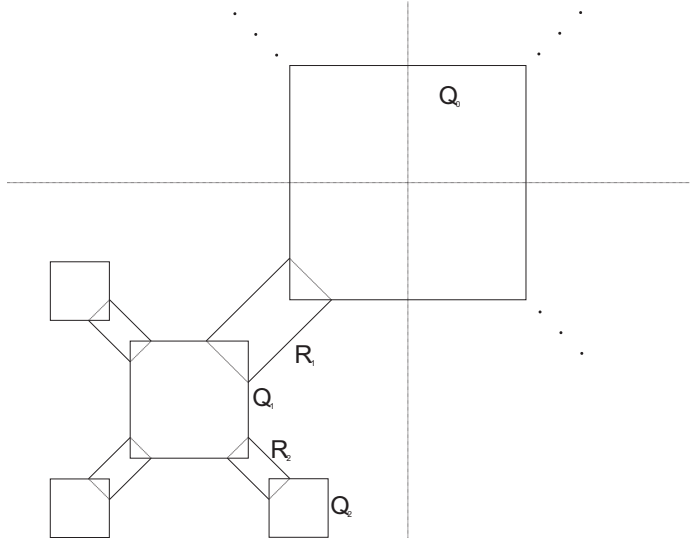}\\
  \caption{The $s$-John domain $\Omega\subset\bR^2$}\label{fig:room and s-passage}
\end{figure}

We choose $E_j$ to be the union of all the cubes at step $j$, \ie the collection of $4\cdot 3^{j-1}$ (disjoint) cubes of edge-length $2^{-j}$. Noting that all the cubes at step $j$ are well separated, to estimate the Hausdorff $q$-content, one has to cover each such cube by a ball of the same size (since otherwise the ball will intersects two cubes and substantially increases the radius). Note also that $q=\log_23$ and so it follows that
\begin{align*}
\mathcal{H}^q_\infty(E_j)\geq C4\cdot 3^{j-1}\cdot 2^{-qj}= C.
\end{align*}  
On the other hand,
\begin{align*}
\capacity_p(E_j,Q_0,\Omega)\leq C4\cdot 3^{j-1}\cdot 2^{-j(s-p+1)}\leq C2^{-j(s-p-q+1)}.
\end{align*}
If $q<s-p+1$, then $\capacity_p(E_j,Q_0,\Omega)\to 0$ as $j\to \infty$ as desired.

Next we consider the case $q<\min\{s-p+1,\log_23\}$. This case is easier and we only need to delete some  ``room and $s$-passage" type ``legs" from the previous construction. To be more precise, we choose $k_j\in \mathbb{N}$ to be an integer such that $k_j-1\leq 2^{qj}\leq k_j$. The construction of the desired $s$-John domain can be proceeded in a similar way. In the first step, we have a unit cube $Q$ and $k_1$ ``room and $s$-passage" type ``legs" as in the previous construction. The ``$s$-passage" $R_1$ is a rectangle of length $2^{-1}$ and width $2^{-s-1}$ and the ``room" $Q_1$ is a cube of edge-length $2^{-1}$. In the second step, we fix $k_2$ corners of all the cubes of edge-length $2^{-1}$ in step 1, and attach at each corner a ``room and $s$-passage" type ``legs". The ``$s$-passage" $R_2$ is a rectangle of length $2^{-2}$ and width $2^{-2s-1}$ and the ``room" $Q_2$ is a cube of edge-length $2^{-2}$. In general at step $j$, we have $k_j$ ``room and $s$-passage" type ``legs", where the ``$s$-passage" $R_j$ is a rectangle of length $2^{-j}$ and width $2^{-js-1}$ and the ``room" $Q_j$ is a cube of edge-length $2^{-j}$.  

Let $E_j$ be the union of all the cubes at step $j$, \ie the collection of $k_j$ (disjoint) cubes of edge-length $2^{-j}$. It is clear that  
\begin{align*}
\mathcal{H}^q_\infty(E_j)\geq Ck_j\cdot 2^{-qj}\geq C.
\end{align*}  
On the other hand, we have
\begin{align*}
\capacity_p(E_j,Q_0,\Omega)\leq Ck_j\cdot 2^{-j(s-p+1)}\leq C2^{-j(s-p-q+1)}.
\end{align*}
If $q<s-p+1$, then $\capacity_p(E_j,Q_0,\Omega)\to 0$ as $j\to \infty$ as desired.

We can construct similar examples in $\bR^n$, $n\geq 3$. Fix $1\leq p\leq n$. Consider the difficult case $q=\log_2(2^n-1)$. The $s$-John domain $\Omega$ will be constructed in a similar manner as before. In the first step, we have a unit cube $Q$ and $2^{n}$ ``room and $s$-passage" type ``legs". The ``$s$-passage" $R_1$ is a cylinder of height $2^{-1}$ and radius $2^{-s-1}$ and the ``room" $Q_1$ is a cube of edge-length $2^{-1}$. In the second step, we attach at each of the $2^n-1$ corners of $Q_1$ a ``room and $s$-passage" type ``legs". The ``$s$-passage" $R_2$ is a cylinder of height $2^{-2}$ and radius $2^{-2s-1}$ and the ``room" $Q_2$ is a cube of edge-length $2^{-2}$. In general at step $j$, we have $2^n\cdot (2^{n}-1)^{j-1}$ ``room and $s$-passage" type ``legs", where the ``$s$-passage" $R_j$ is a cylinder of height $2^{-j}$ and radius $2^{-js-1}$ and the ``room" $Q_j$ is a cube of edge-length $2^{-j}$. It is easy to check that, with our choices of parameters, there is no overlap in our construction. Moreover, $\Omega$ is an $s$-John domain that is Gromov hyperbolic in the quasihyperbolic metric. Indeed, one can easily verify that every quasihyperbolic geodesic triangle in $\Omega$ is $\delta$-thin for some $\delta<\infty$.

We choose $E_j$ to be the union of all the cubes at step $j$, \ie the collection of $2^n\cdot (2^n-1)^{j-1}$ (disjoint) cubes of edge-length $2^{-j}$. Note that $q=\log_2(2^n-1)$ and we obtain that  
\begin{align*}
\mathcal{H}^q_\infty(E_j)\geq C2^n\cdot (2^n-1)^{j-1}\cdot 2^{-qj}= C.
\end{align*}  
On the other hand,
\begin{align*}
\capacity_p(E_j,Q_0,\Omega)&\leq C2^n\cdot (2^n-1)^{j-1}\cdot 2^{-j[(n-1)s-p+1]}\\
&\leq C2^{-j[(n-1)s-p-q+1]}.
\end{align*}
If $q<(n-1)s-p+1$, then $\capacity_p(E_j,Q_0,\Omega)\to 0$ as $j\to \infty$ as desired.

The case $q<\log_2(2^n-1)$ can be proceeded as in the planar case by deleting the extra number of ``room and $s$-passage" type ``legs" and we leave the simple verification to the interested readers. 
\end{example}

\textbf{Acknowledgements}

The author would like to thank his supervisor Academy Professor Pekka Koskela for posing this question and for helpful discussions.

\end{document}